\documentclass[12pt]{gen-j-l}
\usepackage[dvips]{graphicx,color}
\usepackage{amsmath,amssymb,amsthm,amsfonts}
\usepackage[hypertex]{hyperref}
\hypersetup{colorlinks, citecolor=red,pdfstartview=FitH,pdfpagemode=None}

\newtheorem{theorem}{Theorem}[]
\newtheorem{lemma}[theorem]{Lemma}

\newtheorem{thm}[theorem]{Theorem}

\theoremstyle{definition}

\newtheorem{cor}[theorem]{Corollary}

\newcommand{\beq}{\begin{equation}}
\newcommand{\eeq}{\end{equation}}
\newcommand{\BEQ}{\begin{equation*}}
\newcommand{\EEQ}{\end{equation*}}

\newcommand{\bea}{\begin{eqnarray}}
\newcommand{\eea}{\end{eqnarray}}
\newcommand{\BEA}{\begin{eqnarray*}}
\newcommand{\EEA}{\end{eqnarray*}}

\newcommand{\bse}{\begin{subequations}}
\newcommand{\ese}{\end{subequations}}
\newcommand{\BSE}{\begin{subequations*}}
\newcommand{\ESE}{\end{subequations*}}

\newcommand{\bca}{\begin{cases}}
\newcommand{\eca}{\end{cases}}

\newcommand{\ben}{\begin{enumerate}}
\newcommand{\een}{\end{enumerate}}

\newcommand{\bit}{\begin{itemize}}
\newcommand{\eit}{\end{itemize}}

\newcommand{\bex}{\begin{ex}}
\newcommand{\eex}{\end{ex}}

\newcommand{\footer}[1]{{\def\thefootnote{}\footnotetext{#1}}}
\font\germ=eufm10
\def\a{{\mbox{\germ a}}}

\def\g{{\mbox{\germ g}}}

\def\gl{{\mbox{\germ gl}}}
\renewcommand{\sl}{{\mbox{\germ sl}}}

\def\p{{\mbox{\germ p}}}
\def\k{{\mbox{\germ k}}}

\def\R{\mathbb R}
\def\C{\mathbb C}

\def\Z{\mathbb Z}

\def\diag{{\mbox{diag}\,}}

\def\ad{{\rm ad}\,}

\def\Ad{{\rm Ad}\,}

\def\tr{{\rm tr}\,}
\def\GL{{\mbox{\rm GL}}}
\def\SL{{\mbox{\rm SL}}}

\def\Aut{\mbox{\rm Aut}\,}
\def\End{\mbox{\rm End}\,}

\newcommand{\ds}[1]{$\displaystyle{#1}$}

\begin{document}
\title[On Kostant's partial order on hyperbolic elements]{On Kostant's partial order\\ on hyperbolic elements}
\author {Huajun Huang and Sangjib Kim}
\address{Department of Mathematics and Statistics, Auburn University,
AL 36849--5310, USA} \email{huanghu@auburn.edu}

\address{Department of Mathematics, The University of Arizona, Tucson,
AZ 85721--0089, USA} \email{sangjib@math.arizona.edu}

\begin{abstract}
We study Kostant's partial order on the elements of a semisimple Lie
group in relations with the finite dimensional representations. In
particular, we prove the converse statement of \cite[Theorem
6.1]{Ko} on hyperbolic elements.
\end{abstract}

\footer{2000 Mathematics Subject Classification: Primary 22E46} 
\footer{Key Words: Semisimple Lie groups, Complete multiplicative
Jordan decomposition, Hyperbolic elements, Majorizations.}

\maketitle

A matrix in $\GL_n(\C)$ is called {\em elliptic} (resp. {\em
hyperbolic}) if it is diagonalizable  with norm 1 (resp. real
positive) eigenvalues. It is called {\em unipotent} if  all its
eigenvalues  are 1. The {\em complete multiplicative Jordan
decomposition} of   $g\in \GL_n(\C)$ asserts that $g=ehu$ for $e, h,
u\in\GL_n(\C)$,
 where $e$ is elliptic,  $h$ is hyperbolic, $u$ is unipotent,
 and these three elements commute (cf. \cite[p430-431]{He}).
 The decomposition can be easily seen when $g$ is in a
Jordan canonical form: if the diagonal entries (i.e. eigenvalues) of
the Jordan canonical form are $z_1,\cdots, z_n$, then
\beq\label{CMJD-matrix} 
e=\diag\left(\frac{z_1}{|z_1|},\cdots, \frac{z_n}{|z_n|}\right),
\qquad h=\diag\left(|z_1|,\cdots,|z_n|\right), 
\eeq 
and $u=h^{-1}e^{-1}g$ is an upper triangular matrix with  diagonal
entries 1.

The above decomposition can be extended to semisimple Lie groups.
Let $G$ be a connected real semisimple Lie group with Lie algebra
$\g$. An element $e\in G$ is {\em elliptic} if $\Ad e\in \Aut \g$ is
diagonalizable over $\C$ with eigenvalues of modulus $1$.   An
element $h\in G$ is called {\em hyperbolic} if $h = \exp X$ where
$X\in \g$ is real semisimple, that is, $\ad X\in \End \g$ is
diagonalizable over $\R$ with real eigenvalues. An element $u\in G$
is called {\em unipotent} if $u = \exp X$ where $X\in \g$ is
nilpotent, that is, $\ad X\in \End \g$ is nilpotent. The {\em
complete multiplicative Jordan decomposition} \cite [Proposition
2.1]{Ko} for $G$ asserts that each $g\in G$ can be uniquely written
as
\begin{equation}\label{cmjd}
g = ehu,
\end{equation}
where $e$ is elliptic, $h$ is hyperbolic and $u$ is unipotent and
the three elements $e$, $h$, $u$ commute.  We write $g =
e(g)h(g)u(g)$.

When $G=\SL_n(\C)$, the above decomposition in $G$ coincides with
the   complete multiplicative Jordan decomposition in $\GL_n(\C)$
defined in the first paragraph. Let us elaborate on it. Given $g\in
\SL_n(\C)$, there exists $y\in \SL_n(\C)$ such that
$ygy^{-1}\in\SL_n(\C)$ is in a Jordan canonical form. Since the
complete multiplicative Jordan decompositions of $g$ in both
$\GL_n(\C)$ and $\SL_n(\C)$ are preserved by conjugations in
$\SL_n(\C)$, without lost of generality, we may assume that
$g\in\SL_n(\C)$ is already in a Jordan canonical form with diagonal
entries $z_1, z_2,\cdots, z_n$. Then the complete multiplicative
Jordan decomposition of $g$ in $\GL_n(\C)$ is $g=e h u$, where $e$
and $h$ are given in \eqref{CMJD-matrix} and $u=h^{-1}e^{-1}g$ is an
upper triangular matrix with diagonal entries 1. Let
$E_{ij}\in\gl_n(\C)$ denote the matrix with 1 in the $(i,j)$-entry
and 0 elsewhere. Then with respect to the following basis of
$\sl_n(\C)$: 
$$\{E_{ij}\mid i\ne j,\ 1\le i,j\le n\}\cup\{E_{ii}-E_{(i+1)(i+1)}\mid 1\le i\le n-1\},$$ 
the matrix  $\Ad e$ is diagonal with norm 1 eignevalues
$\displaystyle{ {z_iz_j^{-1}}/{|z_iz_j^{-1}|}}$ ($i\ne j$) and 1.
This shows that $e$ is elliptic in $\SL_n(\C)$. Clearly, $h=\exp X$
for a real semisimple element $X:=\log h\in\sl_n(\C)$, and $u=\exp
Y$ for a nilpotent element
$$Y:=\log u=-\frac{I_n-u}{1}-\frac{(I_n-u)^2}{2}-\cdots-\frac{(I_n-u)^{n-1}}{n-1}\in\sl_n(\C),$$ 
where $I_{n}$ is the identity matrix. This shows that $h$ (resp.
$u$) is hyperbolic (resp. unipotent) in $\SL_n(\C)$. Thus $g=ehu$ is
also the complete multiplicative Jordan decomposition of $g$ in
$\SL_n(\C)$.

\medskip

Let us return to a connected real semisimple Lie group $G$ with Lie
algebra $\g$. Fix a Cartan decomposition $\g=\k\oplus\p$ (e.g.,
\cite {He}), and let $\a\subseteq\p$ be a maximal abelian
subalgebra. Let $A:=\exp \a$. Then the hyperbolic component $h(g)$
of $g$ is conjugate to some element $a(g)\in A$. See, for example,
\cite [Proposition 2.4]{Ko}. We denote 
\beq\label{A(g)} A(g):=\exp (\text{conv}(W\cdot\log a(g))), \eeq 
where $W$ is the Weyl group of $(\g, \a)$ and $\text{conv}(W\cdot
\log a(g))$ is the convex hull of the Weyl group orbit of $\log
a(g)$ in $\a$.

Kostant defined a partial order on the elements of $G$ as
\beq\label{partial-order}
 g_1\ge g_2\quad\Longleftrightarrow\quad A(g_1)\supseteq A(g_2),
\eeq and then established the following results:

\begin{theorem}\label{order-cond} \rm \cite[Theorem 3.1]{Ko}
Let $g_1, g_2\in G$. Then $g_1\ge g_2$   if and only if
$|\pi(g_1)|\ge|\pi(g_2)|$ for all finite dimensional representations
$\pi$ of $G$ where $|\cdot|$ denotes the   spectral radius.
\end{theorem}

\begin{theorem}\label{question}\rm
\cite[Theorem 6.1]{Ko} Let $h_1, h_2\in G$ be hyperbolic. Write
$\chi_{\pi}$ for the character of a representation $\pi$ of $G$. If
$h_1\ge h_2$, then $\chi_{\pi}(h_1)\ge\chi_{\pi}(h_2)$ for all
finite dimensional representations $\pi$ of $G$.
\end{theorem}

In \cite [Remark 6.1]{Ko}, Kostant asked if the converse of Theorem
\ref{question} is true or not. Our goal is to   answer  this
question affirmatively and thus we have the following result:

\begin{thm}\label{complete}\rm
  Let $h_1$ and $h_2$ be two hyperbolic elements in a connected real semisimple Lie group $G$.
  Then $h_1\ge h_2$ if and only if $\chi_{\pi}(h_1)\ge \chi_{\pi}(h_2)$ for all finite dimensional representations $\pi$ of $G$.
\end{thm}

Before proving the above result, we need the following two lemmas.

\begin{lemma}\label{diag-matrix}\rm
Let $C$ and $D$ be two diagonal matrices  in $\GL_n(\C)$ with
positive diagonal entries. If the spectral radii $|C|>|D|$, then
there exists a finite dimensional representation $\pi$ of
$\GL_n(\C)$ such that $\chi_{\pi}(C)>\chi_{\pi}(D)$.
\end{lemma}

\begin{proof}
Let $\pi_m$ be the representation of $\GL_n(\C)$ on the $m$-th
symmetric power of the natural representation $\text{Sym}^m(\C^n)$.
Then, the character $\chi_m$ of $\pi_m$ takes the following value on
$\diag(x)=\diag(x_1,x_2,\cdots,x_n)\in\GL_n(\C)$:
$$
\chi_m(\diag(x))=\sum_{\substack{\ell_1, \ell_2, \cdots,\ell_n\ge 0\\ \ell_1+ \cdots+\ell_n=m}} \; \prod_{i=1}^n {x_i^{\ell_i}}.
$$
So $\chi_m(\diag(x))$ is the sum of all $\binom{m+n-1}{n-1}$ monomials of
$x_1,\cdots,x_n$ of degree $m$.
Suppose
$c:=|C|>|D|=: d>0.$
Because
\ds{
\lim_{m\to\infty} (m+n)^{n/m}=1,
}
there exists a sufficiently large $m$ such that
$$
\frac{c}{d}>(m+n)^{n/m}.
$$
Write $C:=\diag(c_1,\cdots,c_n)$ and $D:=\diag(d_1,\cdots,d_n)$. Then
\BEA
\chi_m(C) &=& \sum_{\substack{\ell_1, \ell_2, \cdots,\ell_n\ge 0\\ \ell_1+ \cdots+\ell_n=m}} \; \prod_{i=1}^n c_i^{\ell_i}
\\
&\ge& c^m > (m+n)^n d^m > \binom{m+n-1}{n-1} d^m
\\
&\ge& \sum_{\substack{\ell_1, \ell_2, \cdots,\ell_n\ge 0\\ \ell_1+
\cdots+\ell_n=m}} \; \prod_{i=1}^n d_i^{\ell_i} \ =\ \chi_m(D). \EEA
This completes the proof.
\end{proof}

The following results are shown in \cite[Proposition 3.4 and its
proof]{Ko}: for a representation $\pi:G\to\GL_n(\C)$,
\begin{itemize}
\item
If   $u\in G$ is unipotent, then all the eigenvalues of $\pi(u)$ are
equal to 1;

\item
If $h\in G$ is hyperbolic in $G$, then $\pi(h)$ is diagonalizable
and all the eigenvalues of $\pi(h)$ are positive;

\item
If $e\in G$ is elliptic in $G$, then $\pi(e)$ is diagonalizable and
all the eigenvalues of $\pi(e)$ are of norm 1;

\item
If $g=ehu$ is the complete multiplicative Jordan decomposition of
$g$ in $G$, then  $\pi(g)=\pi(e)\pi(h)\pi(u)$ where $\pi(e)$,
$\pi(h)$, and $\pi(u)$ mutually commute. Hence
$\pi(g)=\pi(e)\pi(h)\pi(u)$ is the complete multiplicative Jordan
decomposition of $\pi(g)$ in $\GL_n(\C)$. 
\end{itemize} 
In addition, if $G$ is connected and semisimple, then $\pi(G)$ is
generated by
$$\exp(d\pi(\g))=\exp([d\pi(\g),d\pi(\g)])\subseteq \exp\sl_n(\C)=\SL_n(\C).$$
Therefore, $\pi(g)\in\SL_n(\C)$ and $\pi(g)=\pi(e)\pi(h)\pi(u)$ is
the complete multiplicative Jordan decomposition of $\pi(g)$ in
$\SL_n(\C)$. Consequently, we have

\begin{lemma}\label{CMJD-preserve}\rm
Let $g$ be an element of a connected real semisimple Lie group $G$,
and $\pi:G\to\GL_n(\C)$ a finite dimensional representation of $G$.
If $g=ehu$ is the complete multiplicative Jordan decomposition of
$g$ in $G$, then $\pi(g)=\pi(e)\pi(h)\pi(u)$ is the complete
multiplicative Jordan decomposition of $\pi(g)$ in $\SL_n(\C)$.
\end{lemma}

\medskip

Now we prove Theorem \ref{complete}:
\begin{proof}[Proof of Theorem \ref{complete}]
The necessary part is Theorem \ref{question}. It remains to prove the sufficient part. From the assumption,
 we have $\chi_{\pi}(h_1)\ge
\chi_{\pi}(h_2)$ for all finite dimensional representations $\pi$ of $G$. Now suppose $h_1 \ngeq h_2$
in $G$. Then Theorem \ref{order-cond}
implies that there exists a finite dimensional representation $\eta: G\to \GL_n(\C)$ of $G$
such that $|\eta(h_1)|<|\eta(h_2)|$. By Lemma
\ref{CMJD-preserve}, both $\eta(h_1)$ and $\eta(h_2)$ are hyperbolic in $\SL_n(\C)$.
So they are conjugate to certain diagonal matrices with
positive diagonal entries. Using Lemma \ref{diag-matrix} and the fact that character
 values are independent of conjugacy, there is a finite
dimensional representation $\pi_m$ of $\GL_n(\C)$ with character
$\chi_m$  such that
$$\chi_m(\eta(h_1))<\chi_m(\eta(h_2)).$$
Now $\rho :=\pi_m\circ\eta$ is a finite dimensional representation
of $G$ with character $\chi_{\rho}$ satisfying that
$$
\chi_{\rho}(h_1)=\chi_m(\eta(h_1))<\chi_m(\eta(h_2))=\chi_{\rho}(h_2).
$$
It contradicts our assumption $\chi_{\pi}(h_1)\ge
\chi_{\pi}(h_2)$ for all $\pi$. 
\end{proof}

Note that  if $g\in\SL_n(\C)$ has the complete multiplicative Jordan
decomposition $g=e(g)h(g)u(g)$   in $\SL_n(\C)$, then the hyperbolic
component $h(g)$ of $g$ is diagonalizable and its eigenvalues are
equal to the eigenvalue moduli of $g$.

Given a finite dimensional representation $\pi$ of $G$ with
character $\chi_{\pi}$, let us denote by $|\chi_{\pi}|(g)$ the sum
of eigenvalue moduli of $\pi(g)$ for $g\in G$. Then, from the above
observation, we have \beq |\chi_{\pi}|(g):=\tr (h(\pi(g))). \eeq
Then Theorem \ref{complete} can be extended to an equivalent
condition for the partial ordering on {\em all elements} of $G$.

\begin{cor}\label{complete-g}\rm
 Let $G$ be a connected real semisimple Lie group  and $g_1, g_2\in G$.
 Then $g_1\ge g_2$ in $G$ if and only if $|\chi_{\pi}|(g_1)\ge |\chi_{\pi}|(g_2)$
 for all finite dimensional representations $\pi$ of $G$.
\end{cor}

\begin{proof}
By the definitions \eqref{A(g)} and \eqref{partial-order} of the
partial order, $g_1\ge g_2$ in $G$ if and only if $h(g_1)\ge h(g_2)$
in $G$; if and only if $\chi_{\pi}(h(g_1))\ge \chi_{\pi}(h(g_2))$
for all finite dimensional representations $\pi$ of $G$ (Theorem
\ref{complete}), namely
$$
\tr(\pi(h(g_1)))\ge\tr(\pi(h(g_2))).
$$
According to Lemma \ref{CMJD-preserve}, $\pi(h(g))=h(\pi(g))$ for
all $g\in G$. Thus $\tr(\pi(h(g_1)))\ge\tr(\pi(h(g_2)))$ if and only
if $\tr(h(\pi(g_1)))\ge\tr(h(\pi(g_2)))$, namely
$|\chi_{\pi}|(g_1)\ge |\chi_{\pi}|(g_2)$.
\end{proof}

For a finite dimensional representation $\pi:G\to \GL_{n}(\C)$, we
denote by \beq \lambda_{\pi}^{(1)}(g)\ge
\lambda_{\pi}^{(2)}(g)\ge\cdots\ge \lambda_{\pi}^{(n)}(g)>0 \eeq the
eigenvalue moduli of $\pi(g)$ in non-increasing order for $g\in G$,
that is, $\lambda_{\pi}^{(1)}(g)\ge \cdots\ge
\lambda_{\pi}^{(n)}(g)$ are the eigenvalues of
$h(\pi(g))=\pi(h(g))$. Then Theorem \ref{order-cond} says that
$g_1\ge g_2$ in $G$ if and only if $\lambda_{\pi}^{(1)}(g_1)\ge
\lambda_{\pi}^{(1)}(g_2)$ for all finite dimensional representations
$\pi$ of $G$, and Corollary \ref{complete-g} says that $g_1\ge g_2$
in $G$ if and only if $\sum_{i=1}^n \lambda_{\pi}^{(i)}(g_1)\ge
\sum_{i=1}^n \lambda_{\pi}^{(i)}(g_2)$ for all finite dimensional
representations $\pi: G\to \GL_{n}(\C)$. In general, we have the
following result.

\begin{theorem}\label{majorization}\rm
Let $G$ be a connected real semisimple Lie group  and $g_1, g_2\in
G$.

\ben

\item
If $g_1\ge g_2$ in $G$, then for every finite dimensional
representation $\pi: G\to \GL_{n}(\C)$, the following inequalities
hold for $k=1,\cdots, n$:
\BEQ \prod_{i=1}^{k} \lambda_{\pi}^{(i)}(g_1) \ge \prod_{i=1}^{k}
\lambda_{\pi}^{(i)}(g_2),\qquad \sum_{i=1}^{k}
\lambda_{\pi}^{(i)}(g_1) \ge \sum_{i=1}^{k}
\lambda_{\pi}^{(i)}(g_2). 
\EEQ 

\item
Fix $k\in\Z^+$. For every finite dimensional representation $\pi:
G\to \GL_{n}(\C)$ with $n\ge k$, if we have 
\BEQ  
\sum_{i=1}^{k} \lambda_{\pi}^{(i)}(g_1) \ge \sum_{i=1}^{k}
\lambda_{\pi}^{(i)}(g_2), 
\EEQ 
then $g_1\ge g_2$ in $G$.

\item
Fix $k\in\Z^+$. For every finite dimensional representation $\pi:
G\to \GL_{n}(\C)$  with $n\ge k$, if we have 
\BEQ  
\prod_{i=1}^{k} \lambda_{\pi}^{(i)}(g_1) \ge \prod_{i=1}^{k}
\lambda_{\pi}^{(i)}(g_2), 
\EEQ 
then $g_1\ge g_2$ in $G$. \een
\end{theorem}

\begin{proof}
\text{} 
\ben 
\item Suppose that $g_1\ge g_2$ in $G$. Given any finite dimensional
representation $\pi:G\to\GL_{n}(\C)$, let  $\rho_k: \GL_{n}(\C)\to
\GL(\wedge^k \C^n)$ denote the fundamental representation of
$\GL_{n}(\C)$ on the $k$-th exterior power of $\C^n$,
$k=1,\cdots,n$. Then by Theorem \ref{order-cond},   
\BEQ 
|\rho_k(\pi(g_1))| \ge |\rho_k(\pi(g_2))|,\qquad k=1,\cdots,n-1, 
\EEQ 
and by $\pi(G)\subseteq \SL_{n}(\C)$,  
$$\qquad
|\rho_n(\pi(g_1))|
=|\det(\pi(g_1))|=1=|\det(\pi(g_2))|=|\rho_n(\pi(g_2))|.
$$ 
Equivalently,
\BEA 
\qquad \prod_{i=1}^k \lambda_{\pi}^{(i)}(g_1) &\ge & \prod_{i=1}^k
\lambda_{\pi}^{(i)}(g_2),\qquad k = 1, 2, \cdots,n-1,
\\
\prod_{i=1}^n \lambda_{\pi}^{(i)}(g_1) &=& \prod_{i=1}^n
\lambda_{\pi}^{(i)}(g_2)=1. \EEA  
So the vector of  eigenvalue moduli of $\pi(g_1)$ multiplicatively
majorizes that of $\pi(g_2)$. Moreover,
 multiplicative majorization
of vectors with real positive numbers implies additive majorization
\cite[Example II.3.5(vii)]{Bh}. Hence
\BEA \sum_{i=1}^k \lambda_{\pi}^{(i)}(g_1) &\ge & \sum_{i=1}^k
\lambda_{\pi}^{(i)}(g_2),\qquad k = 1, 2, \cdots,n. \EEA 

\item
Fix $k\in\Z^+$. Let $\pi:G\to\GL_{n}(\C)$ be an arbitrary finite
dimensional representation of $G$.  Let $k\pi$ denote the
representation formed by taking the direct sum of $k$ copies of
$\pi$. Then the representation space of $k\pi$ has dimension $kn \ge
k$. By assumption,
$$k\lambda_{\pi}^{(1)}(g_1) = \sum_{i=1}^k \lambda_{k\pi}^{(i)}(g_1)
\ge \sum_{i=1}^k \lambda_{k\pi}^{(i)}(g_2) =
k\lambda_{\pi}^{(1)}(g_2).$$ 
So $|\pi(g_1)|\ge |\pi(g_2)|$. Thus $g_1\ge g_2$ in $G$ by Theorem
\ref{order-cond}. 

\item
The proof is similar to that of (2). 
\een
\end{proof}

\vspace{3mm}
\begin{flushleft}
{\bf Acknowledgement:} We thank Tin-Yau Tam for helpful discussions.

\end{flushleft}


\begin{thebibliography}{99}

\bibitem{Bh} R. Bhatia, {\em Matrix Analysis}, Springer, New York, 1997.

\bibitem{He} S. Helgason, {\em Differential Geometry, Lie Groups, and Symmetric Spaces}, American Mathematical Society, 2001.

\bibitem{Ko} B. Kostant,  {\em On convexity, the Weyl group and the Iwasawa decomposition},
 Ann. Sci. \'{E}cole Norm. Sup. (4),  {\bf 6}  (1973) 413--455.


\end{thebibliography}
\end{document}